\documentclass{amsart}
\usepackage{amssymb}
\usepackage{amsfonts}
\usepackage{latexsym}
\usepackage{extpfeil}
\usepackage{mathrsfs}

\newtheorem{theorem}{Theorem}[section]
\newtheorem{lemma}[theorem]{Lemma}
\newtheorem{proposition}[theorem]{Proposition}
\newtheorem{corollary}[theorem]{Corollary}
\theoremstyle{definition}

\newtheorem{example}[theorem]{Example}

\newtheorem{remark}[theorem]{Remark}

\newcommand{\id}{\text{id}}

\newcommand{\FPdim}{\text{\rm FPdim}}

\newcommand{\Hom}{\text{Hom}}

\newcommand{\gr}{{\rm gr}}

\newcommand{\Rep}{{\rm Rep}}
\newcommand{\Cyc}{{\rm Cyc}}
\newcommand{\Alt}{{\rm Alt}}

\newcommand{\ot}{\otimes}

\newcommand{\ben}{\begin{enumerate}}
\newcommand{\een}{\end{enumerate}}

\newcommand{\Rad}{{\rm Rad}}
\newcommand{\Vect}{{\rm Vec}}

\newcommand{\dd}{{\rm d}}

\numberwithin{equation}{section}

\begin{document}

\title[symmetric tensor categories with the Chevalley property]{Finite symmetric tensor categories with the Chevalley property in characteristic $2$}

\author{Pavel Etingof}
\address{Department of Mathematics, Massachusetts Institute of Technology,
Cambridge, MA 02139, USA} \email{etingof@math.mit.edu}

\author{Shlomo Gelaki}
\address{Department of Mathematics, Iowa State University, Ames, IA 50011, USA} \email{gelaki@iastate.edu}

\date{\today}


\dedicatory{Dedicated to Nicol\'{a}s Andruskiewitsch for his 60th birthday}

\begin{abstract}
We prove an analog of Deligne's theorem for finite symmetric tensor categories $\mathcal{C}$ with the Chevalley property over an algebraically closed field $k$ of characteristic $2$. Namely, we prove that every such category $\mathcal{C}$ admits a symmetric fiber functor to the symmetric tensor category $\mathcal{D}$ of representations of the triangular Hopf algebra $(k[\dd]/(\dd^2),1\ot 1 + \dd\ot \dd)$. Equivalently, we prove that there exists a unique finite group scheme $G$ in $\mathcal{D}$ 
such that $\mathcal{C}$ is symmetric tensor equivalent to $\Rep_{\mathcal{D}}(G)$. Finally, we compute the group $H^2_{\rm inv}(A,K)$ of equivalence classes of twists for the group algebra $K[A]$ of a finite abelian $p$-group $A$ over an arbitrary field $K$ of characteristic $p>0$, and the Sweedler cohomology groups $H^i_{\rm{Sw}}(\mathcal{O}(A),K)$, $i\ge 1$, of the function algebra $\mathcal{O}(A)$ of $A$. 
\end{abstract}

\maketitle

\section{Introduction}

The main objective of this paper is to classify finite symmetric tensor categories with the Chevalley property over an algebraically closed field $k$ of characteristic $2$. This completes the classification of finite integral symmetric tensor categories with the Chevalley property over an algebraically closed field of characteristic $p>0$, which  for $p>2$ was established in \cite{eg5}, since by \cite[Theorem 1.5]{o}, integrality follows from the rest of the conditions for $p=2,3$. 
 
Let $\alpha_2$ be the Frobenius kernel of the additive group $\mathbb{G}_a$. Then $k\alpha_2=k[\dd]/(\dd^2)$ with $\dd$ primitive. Let $\mathcal{D}:=\Rep(\alpha_2,1\ot 1 + \dd\ot \dd)$ be the symmetric tensor category of finite dimensional representations of the triangular Hopf algebra $k[\dd]/(\dd^2)$ equipped with the $R$-matrix $1\ot 1 + \dd\ot \dd$   
\footnote{$\mathcal{D}$ 
may be considered as a non-semisimple analog in characteristic $2$ of the category $\text{sVec}$ of supervector spaces, see \cite{v}.}. Recall \cite{v} that an object in $\mathcal{D}$ is a finite dimensional $k$-vector space $V$ together with a linear map  $\dd:V\to V$ satisfying $\dd^2=0$. In particular, $\mathcal{D}$ has two indecomposable objects, namely, the unit object (i.e., the vector space $k$ with $\dd=0$), and the two dimensional vector space $k^2$ with $\dd$ the strictly upper triangular matrix $E_{12}$.

Recall that a finite group scheme in $\mathcal{D}$ is, by definition, a finite dimensional {\em cocommutative} Hopf algebra $H$ in $\mathcal{D}$. In particular, this means that $\dd:H\to H$ is a derivation of $H$ satisfying $\dd^2=0$, and 
$$
\Delta(h)=(1\ot 1 + \dd\ot \dd)(\Delta(h))_{21},\,\,\,h\in H.
$$

We can now state our main result (compare with  \cite[Conjecture 1.3]{o}).

\begin{theorem}\label{classuniptr2}
Let $\mathcal{C}$ be a finite symmetric tensor category with the Chevalley property over an algebraically closed field $k$ of characteristic $2$. Then $\mathcal{C}$ admits a symmetric fiber functor to $\mathcal{D}$. Thus, there exists a unique finite group scheme $G$ in $\mathcal{D}$ 
such that $\mathcal{C}$ is symmetric tensor equivalent to the category $\Rep_{\mathcal{D}}(G)$ of finite dimensional representations of $G$ which are compatible with the action of $\pi_1(\mathcal{D})$. 
\end{theorem}
%
%

\begin{remark}
Theorem \ref{classuniptr2} answers \cite[Question 1.2]{be} for finite symmetric tensor categories with the Chevalley property over $k$, and we expect it to hold for every finite symmetric integral tensor category over $k$.
\end{remark}

Finally, we note that the arguments used to prove \cite[Theorem 1.1]{eg5} and Theorem \ref{classuniptr2} in fact prove a stronger result (see Theorem \ref{stronger}).

The organization of the paper is as follows. Section 2 is devoted to the proof of Theorem \ref{classuniptr2}. In Section 3 we compute the group $H^2_{\rm inv}(A,K)$ of equivalence classes of twists for the group algebra $K[A]$ of a finite abelian $p$-group $A$ over an arbitrary field $K$ of characteristic $p>0$ (see Theorem \ref{invtwpabconc}), and use it together with  \cite[Proposition 5.7]{eg5} to compute the Sweedler cohomology groups $H^i_{\rm{Sw}}(\mathcal{O}(A),K)$ for every $i\ge 1$ (see Theorem \ref{sweedcoh}).

{\bf Acknowledgments.} P. E. was partially supported by NSF grant DMS 1502244. S. G. is grateful to the University of Michigan and MIT for their warm hospitality.

\section{The proof of Theorem \ref{classuniptr2}}\label{char2}
All constructions in this section are done over an algebraically closed field $k$ of characteristic $2$ unless otherwise is explicitly stated. To lighten notation, we sometimes write $1$ for $1\otimes 1$ or $1\otimes 1\otimes 1$. 

We refer the reader to \cite{egno} for the general theory of finite tensor categories, to \cite{d} for generalities on quasi-Hopf algebras (see also \cite[2.1]{eg5}), and to \cite{j,w} for the general theory of finite group schemes (see also \cite[2.4]{eg5}).

By \cite[Theorem 1.5]{o}, any finite symmetric tensor category with the Chevalley property in characteristic $2$  is integral (as ${\rm Ver}_2=\Vect$). Thus 
by \cite[Theorem 2.6]{eo}, $\mathcal{C}$ is symmetric tensor equivalent to $\Rep(H,R,\Phi)$ for some finite dimensional  triangular quasi-Hopf algebra $(H,R,\Phi)$ with the Chevalley property over $k$. Thus, we have to prove the following theorem.

\begin{theorem}\label{symmwithphimain2}
Let $(H,R,\Phi)$ be a finite dimensional triangular quasi-Hopf algebra with the Chevalley property over $k$. Then $(H,R,\Phi)$ is pseudotwist equivalent to a triangular Hopf algebra with $R$-matrix $1+\dd\ot \dd$ for some $\dd\in P(H)$ such that $\dd^2=0$.
\end{theorem}

We will prove Theorem \ref{symmwithphimain2} in several steps.

\subsection{$\gr(H)$}\label{2.1} Let $(H,R,\Phi)$ be a finite dimensional triangular quasi-Hopf algebra with the Chevalley property over $k$. Let $I:=\Rad(H)$ be the Jacobson radical of $H$. Since $I$ is a quasi-Hopf ideal of $H$, the associated graded algebra $\gr(H)=\bigoplus_{r\ge 0} H[r]$ has a natural structure of a graded triangular quasi-Hopf algebra with some $R$-matrix $R_0\in H[0]^{\otimes 2}$ and associator $\Phi_0\in H[0]^{\otimes 3}$ (see, e.g., \cite[2.2]{eg5}).

\begin{proposition} \cite[Proposition 3.2]{eg5}\label{main1}
The following hold:
\begin{enumerate}
\item
$H[0]$ is semisimple.
\item
$(H[0],R_0,\Phi_0)$ is a triangular quasi-Hopf subalgebra of $(\gr(H),R_0,\Phi_0)$.
\item
$\Rep(H[0],R_0,\Phi_0)$ is symmetric tensor equivalent to $\Rep(G)$ for some finite semisimple group scheme $G$ over $k$.
\item
$(\gr(H),R_0,\Phi_0)$ is pseudotwist equivalent to a graded triangular Hopf algebra with $R$-matrix $1\ot 1$, whose degree $0$-component is $(kG,1\ot 1)$. \qed
\end{enumerate}
\end{proposition}

\begin{corollary} \cite[Corollary 3.3]{eg5}\label{gr}
Let $(H,R,\Phi)$ be a finite dimensional triangular quasi-Hopf algebra with the Chevalley property over $k$. 
Then $\gr(H)$ is pseudotwist equivalent to $k\mathcal{G}$ for some finite group scheme $\mathcal{G}$ over $k$ containing $G$ as a closed subgroup scheme. \qed
\end{corollary}

\begin{remark}\label{epsilon=1}
By Nagata's theorem (see, e.g, \cite[p.223]{a}), we have $G=\Gamma\ltimes P^D$, where $\Gamma$ is a finite group of odd order and $P$ is a finite abelian $2$-group. Hence, we have $kG=k\Gamma\ltimes \mathcal{O}(P)$.
\end{remark}

Let $\Gamma:=\mathcal{G}/\mathcal{G}^{\circ}$. Then $\Gamma$ is a finite constant group of odd order, and we have $\mathcal{G}=\Gamma\ltimes \mathcal{G}^{\circ}$. Thus, we have $\mathcal{O}(\mathcal{G})=\mathcal{O}(\Gamma)\otimes \mathcal{O}(\mathcal{G}^{\circ})$ as algebras.

By the results of this subsection, we may assume without loss of generality in the proof of Theorem \ref{symmwithphimain2} that $R=1+$ terms of higher degree.

\subsection{Trivializing $R$}\label{2.2}

Let $V$ be a $k$-vector space, and let $\tau:V^{\ot 2}\to V^{\ot 2}$ be the flip map. Recall that 
$$\wedge^2V:=\text{Im}(\id +\tau)\subset \Gamma^2V:=\text{Ker}(\id +\tau)\subset V^{\ot 2},$$
$$S^2V:=V^{\ot 2}/\wedge^2V,\,\,\,\,\,V^{(1)}:=\Gamma^2V/\wedge^2V,$$
and that $V^{(1)}$ is called the Frobenius twist of $V$ and $\Gamma^2V$ the divided second symmetric power of $V$. Note that $V^{(1)}$ is the image of the composition
$$\Gamma^2V\hookrightarrow V^{\ot 2}\twoheadrightarrow S^2V.$$
Let $\pi: \Gamma^2V\to V^{(1)}$ be the natural surjective map. 

Let $(H,R,\Phi)$ be as in the end of Section \ref{2.1}.

\begin{proposition}\label{trivR} 
The following hold:
\begin{enumerate}
\item
Suppose $R=1+\dd_{n-1}\ot \dd_{n-1}$ modulo terms of degree $\ge n\ge 1$ such that $\dd_{n-1}\in \Rad(H)$. 
Then $(H,R,\Phi)$ can be twisted to a form such that $R=1+\dd_n\otimes \dd_n$ modulo terms of degree $\ge n+1$, where $\dd_{n}\in \Rad(H)$ and $\dd_n-\dd_{n-1}$ has degree $\ge n/2$, by a pseudotwist $J_n$ such that $J_n-1$ has degree $\ge n$ if $\dd_{n-1}=0$, and degree $\ge \frac{n}{2}+p$ if ${\rm deg}(\dd_{n-1})=p>0$. 
\item
If $R\ne 1$ then $(H,R,\Phi)$ can be twisted to the form $R=1+\dd\ot \dd$, where $\dd\in \Rad(H)$ is an element of positive degree. Moreover, if $R=1+\dd'\ot \dd'$ modulo terms of degree $\ge n$, where   
$\dd'\in \Rad(H)$,
then this can be achieved by a pseudotwist $J$ with $J-1$ of degree $\ge n$ if $\dd'=0$, and degree  
$\ge \frac{n}{2}+p$ if $\dd'\ne 0$ and has degree $p$, so that $\dd-\dd'$ has degree $\ge n/2$. 
\item
If $R=1+\dd\ot \dd$ then $\dd^2=0$.
\end{enumerate} 
\end{proposition}

\begin{proof} 
(1) Let $R=1+\dd_{n-1}\ot \dd_{n-1}$ modulo terms of degree $\ge n$, and consider $R$ modulo terms of degree $\ge n+1$. We have $R=1+\dd_{n-1}\ot \dd_{n-1}+\widetilde{s}$ modulo terms of degree $\ge n+1$, where $\widetilde{s}\in H^{\otimes 2}$ has degree $\ge n$. Let $s\in \gr(H)^{\ot 2}[n]$ be the leading part of $\widetilde{s}$. Then $s$ is symmetric because $R_{21}R=1\otimes 1$, so $s\in \Gamma^2\gr(H)[n]$. Moreover, if $t\in \wedge^2\gr(H)[n]$ then we can replace $s$ by $s+t$ by twisting. 

Let $v:=\pi(s)$ be the image of $s$ in $\gr(H)^{(1)}[n]=\gr(H)[n/2]^{(1)}$ (note that this space can be nonzero only if $n$ is even). Then we can twist $s$ into the form $v\ot v$ by a pseudotwist $J$ with $J-1$ of degree $\ge n$. So we will get $R=1+\dd_{n-1}\ot \dd_{n-1}+\widetilde{v}\ot \widetilde{v}$ modulo terms of degree $\ge n+1$, where $\widetilde{v}$ is a lift of $v$ to $H$. 
If $\dd_{n-1}=0$, this completes the proof (we can set $\dd_n=\widetilde{v}$). Thus, it remains to consider the case when $\dd_{n-1}\ne 0$ and has degree $p$; so we may assume that $n>2p$ (because for $n\le 2p$, we can set $\dd_n=\dd_{n-1}$ and $J=1$). In this case, let us twist by $J=1+\dd_{n-1}\ot \widetilde{v}$ (note that $\deg(J-1)\ge \frac{n}{2}+p$). Since $R_{21}R=1\ot 1$, we have $\text{deg}(\dd_{n-1}^2)\ge n/2$, hence 
$$
\text{deg}(\dd_{n-1}^2\ot \dd_{n-1}\widetilde{v})\ge n/2+n/2+p=n+p\ge n+1,
$$ 
so twisting by $J$ brings $R$ to the form $R=1+(\dd_{n-1}+\widetilde{v})\ot (\dd_{n-1}+\widetilde{v})$ modulo terms of degree $\ge n+1$, i.e., we may take $\dd_n=\dd_{n-1}+\widetilde{v}$, as desired. 

(2) Follows immediately from (1). Namely, for the first statement we take $\dd$ to be the stable limit of the $\dd_m$'s and $J$ to be the product of the $J_m$'s, and for the second statement we take $\dd_{n-1}=\dd'$, $\dd$ to be the stable limit of the $\dd_m$'s, and $J$ to be the product of the $J_m$'s for $m\ge n$. 

(3) Follows from the identity $R_{21}R=1\otimes 1$. 
\end{proof}

Thus, from now on we may assume that $R=1+\dd\ot \dd$ for some $\dd\in\Rad(H)$ with $\dd^2=0$ (but in general $\dd$ is not a primitive element yet, as we have not made $\Phi=1$).

\begin{remark}\label{stable}
Proposition \ref{trivR} implies that the degree $p$ of $\dd$ in Proposition \ref{trivR}(2) and its degree $p$ part $\delta\in\gr(H)[p]$ (when $\dd\ne 0$) are uniquely determined. Indeed, if 
$(H,R,\Phi)$ is pseudotwist equivalent to $(H,R',\Phi')$ where 
$R=1+\dd\ot \dd$ and $R'=1+\dd'\ot \dd'$ modulo terms of degree $\ge n$, 
and if $\dd\ne 0$ and has degree $p<n/2$, then by Proposition \ref{trivR}(1)
the pseudotwist $J$ can be chosen so that $J-1$ is of degree \linebreak $\ge \frac{n}{2}+p>2p$, so 
$\dd'-\dd$ has degree $\ge p+1$, as desired. In particular, if $R=J_{21}^{-1}J$ then whenever $R$ is twisted to $1+ \dd\ot \dd$, we must have $\dd=0$. This is the case when $\Rep(H,R,\Phi)$ is Tannakian (as follows from Theorem \ref{symmwithphimain2}). However, $\dd$ itself is not unique (e.g., it can be conjugated by an invertible element $x$ of $1+{\rm Rad}(H)$, which results from applying the coboundary twist attached to $x$).
\end{remark}

\subsection{Trivializing $\Phi$} 
Let $(H,R,\Phi)$ be a finite dimensional triangular quasi-Hopf algebra with the Chevalley property over $k$, where $R=1+\dd\otimes \dd$ for some element $\dd\in\Rad(H)$ with $\dd^2=0$. 

By Corollary \ref{gr}, $\gr(H)=k[\mathcal{G}]=\bigoplus_{i\ge 0} k[\mathcal{G}][i]$, as graded Hopf algebras, for some finite group scheme $\mathcal{G}$ over $k$. We let $m$, $\varepsilon$ denote the multiplication and counit maps of $\mathcal{O}(\mathcal{G})$.

If $\Phi=1$ then $\dd^2=0$ and $\Delta(\dd)=\dd\ot 1+1\ot \dd$, so we are done. Thus we may assume that $\Phi\ne 1$. Consider $\Phi-1$. If it has degree $\ell$ then let $\phi$ be its projection to $\gr(H)^{\ot 3}[\ell]$. 

For every permutation $(i_1i_2i_3)$ of $(123)$, we will use $\phi_{i_1i_2i_3}$ to denote the $3$-tensor obtained by permuting the components of $\phi$ accordingly.

\begin{lemma}\label{qtalt} 
The following hold:

\begin{enumerate}
\item 
$\phi\in Z^3(\mathcal{O}(\mathcal{G}),k)$ is a normalized Hochschild $3$-cocycle of $\mathcal{O}(\mathcal{G})$ with coefficients in the trivial module $k$, i.e.,  
$$\phi\circ ({\rm id}\ot {\rm id}\ot m)+\phi\circ (m\ot {\rm id}\ot {\rm id})= \varepsilon\ot \phi +\phi\circ ({\rm id}\ot m\ot {\rm id})+ \phi\ot \varepsilon$$
and
$$
\phi\circ ({\rm id}\ot {\rm id}\ot 1)=\phi\circ (1\ot {\rm id}\ot {\rm id})=\phi\circ ({\rm id}\ot 1\ot {\rm id})=\varepsilon \ot \varepsilon.$$ 
\item 
$\Alt(\phi):=\phi_{312}+\phi_{132}+\phi_{123} +\phi_{231}+\phi_{213}+\phi_{321}=0$.
\end{enumerate}
\end{lemma}

\begin{proof}
(1) Follows from \cite[(2.1)-(2.2)]{eg5} in a straightforward manner.

(2) Follows from \cite[(2.8)]{eg5} in a straightforward manner.
\end{proof}

\subsubsection{The case $R=1\ot 1$.}
In this subsection we will assume that $R=1\ot 1$, i.e., $\dd=0$. 

\begin{lemma}\label{qtalt1} 
The following hold:
\begin{enumerate}
\item
$\phi_{312}+\phi_{132}+\phi_{123}=0=\phi_{231}+\phi_{213}+\phi_{123}$.
\item
$\phi_{123}=\phi_{321}$.
\item
$\Cyc(\phi):=\phi_{312}+\phi_{231}+\phi_{123}=0$.
\end{enumerate}
\end{lemma}

\begin{proof}
(1) Follows from \cite[(2.6)-(2.7)]{eg5} in a  straightforward manner.

(2) Using (1) and Lemma \ref{qtalt}(2), we get
\begin{eqnarray*}
\lefteqn{0=\phi_{312}+\phi_{132}+\phi_{123}+(\phi_{231}+\phi_{213}+\phi_{123})}\\
& = & \Alt(\phi)+\phi_{321}+\phi_{123}\\
& = & \phi_{123}+\phi_{321},
\end{eqnarray*}
as claimed.

(3) By (2), we have $\phi_{132}=\phi_{231}$. Thus the claim follows from (1).
\end{proof}

Following \cite[2.8]{eg5}\footnote{In \cite[2.8]{eg5}, $y_t^{(l)}$ was denoted by $(x_t^l)^*$.}, we set $y_t:=x_t^*$ and $y_t^{(l)}:=(x_t^l)^*$, $1\le t\le n$, $1\le l\le r_t-1$ (so, $y_t^{(1)}=y_t$), and for every $1\le i,j\le n$, let
\begin{equation}\label{betagamma}
\beta_j:=\sum_{l=1}^{2^{r_j}-1}y_j^{(l)}\ot y_j^{(2^{r_j}-l)}.
\end{equation} 

\begin{proposition}\label{killingPhi}
The $3$-cocycle $\phi$ is a coboundary.
\end{proposition}

\begin{proof} 
By Lemma \ref{qtalt}(1), $\phi\in Z^3(\mathcal{O}(\mathcal{G}),k)$ so we can express it in the following form:
$$
\phi=\sum_{1\le i< j< l\le n}b_{ijl}(y_i\ot y_j\ot y_l)+
\sum_{i,j}a_{ij}\beta_i\otimes y_j + df,$$
for some $b_{ijl},a_{ij}\in k$ and   
$f\in k[\mathcal{G}]^{\ot 2}$. 

Thus by Lemma \ref{qtalt1}(3), we have
\begin{eqnarray*}
\Cyc(df) & = & \sum_{1\le i<j<l\le n}b_{ijl}\text{Cyc}(y_i\ot y_j\ot y_l)+
\sum_{i,j}a_{ij}\Cyc(\beta_i\otimes y_j).
\end{eqnarray*}
Also, since $\Alt(df)=\Alt(\beta_i\otimes y_j)=0$, it follows from Lemma \ref{qtalt}(2) and the above that we have 
\begin{equation*}
0=\Alt(\phi)=\sum_{1\le i<j<l\le n}b_{ijl}\Alt(y_i\ot y_j\ot y_l).
\end{equation*}
Therefore $b_{ijl}=0$ for every $i<j<l$, and we have
\begin{eqnarray*}
\Cyc(df) & = & 
\sum_{i,j}a_{ij}\Cyc(\beta_i\otimes y_j).
\end{eqnarray*}

It is also straightforward to verify that we have
\begin{equation}\label{star}
\Cyc(df)=\Cyc\left((\Delta\ot \id)(f)+(\id\ot \Delta)(f)\right).
\end{equation}

Consider the surjective homomorphism 
$$
\Psi:=\pi\ot \id:\Gamma^2k[\mathcal{G}]\ot k[\mathcal{G}]\to k[\mathcal{G}]^{(1)}\ot k[\mathcal{G}],
$$
where $\pi:\Gamma^2k[\mathcal{G}]\to k[\mathcal{G}]^{(1)}$ is the natural surjective homomorphism. Observe that we have  $\pi(\Delta(u))=0$ for every $u\in \Rad(H)$. Indeed this holds for  $u=u_1\cdots u_m$, where $u_1,\dots, u_m$ are primitive, and each element of $\Rad(H)$ is a linear combination of such with coefficients in $G$. 

Now since $k[\mathcal{G}]$ is cocommutative, it follows from (\ref{star}) that $\Cyc(df)$ is symmetric, hence we have
\begin{equation}\label{2star}
\Psi(\Cyc(df))=0.
\end{equation}
We also have
$$
\Psi(\beta_i\ot y_j)=y_i^{(2^{r_i-1})}\ot y_j.
$$
Thus
$$
\sum_{i,j}a_{ij}y_i^{(2^{r_i-1})}\otimes y_j=0,
$$
which implies that $a_{ij}=0$ for all $i,j$. Thus  
$\phi=df$ is a coboundary, as claimed.
\end{proof}

\begin{lemma}\label{symm1} 
In Proposition \ref{killingPhi} we can choose $f\in \Gamma^2k[\mathcal{G}]$, i.e., we can choose $f$ to be  symmetric. 
\end{lemma} 

\begin{proof} Since $\phi_{123}=\phi_{321}$ by Lemma \ref{qtalt1}(2), we have $df=d(f_{21})$. This implies that  $f+f_{21}\in Z^{2}(\mathcal{O}(\mathcal{G}),k)$ is a $2$-cocycle, so it follows from \cite[Proposition 2.4(2)]{eg5} that we have
$$
f+f_{21}=\sum_i a_i\beta_i+\sum_{i< j}b_{ij}(y_i\ot y_j) +z\ot 1+1\ot z+\Delta(z)
$$   
for some $a_i,b_{ij}\in k$ and $z\in k[\mathcal{G}]$. Since the left hand side is symmetric and $\Delta=\Delta^{cop}$, we must have $b_{ij}=0$ for every $i<j$. Applying the map $\Psi$ then yields $a_i=0$ for every $i$. Thus, we have
\begin{equation}\label{f21}
f+f_{21}=z\ot 1+1\ot z+\Delta(z).
\end{equation}
Hence, applying the operator $y\mapsto y\otimes 1+1\otimes y+\Delta(y)$ to the first tensorand, 
we get 
$$
f_{12,3}+f_{1,3}+f_{2,3}+f_{3,12}+f_{3,1}+f_{3,2}=
z_{123}+z_{12}+z_{23}+z_{13}+z_1+z_2+z_3.
$$
Hence, the left hand side is symmetric, so 
$$
\Cyc(df)=\Cyc(f_{12,3}+f_{1,3}+f_{2,3}+f_{3,12}+f_{3,1}+f_{3,2})=
z_{123}+z_{12}+z_{23}+z_{13}+z_1+z_2+z_3.
$$
Since $\Cyc(df)=0$, this implies that 
\begin{equation}\label{z}
z_{123}+z_{12}+z_{23}+z_{13}+z_1+z_2+z_3=0. 
\end{equation}
Let $w:=z_{1,2}+z_1+z_2$. 
Then Equation (\ref{z}) implies
$$
w_{12,3}+w_{1,3}+w_{2,3}=0. 
$$
This means that the tensorands of $w$ are primitive, hence $w=\sum_{i, j} c_{ij}p_i\otimes p_j$, where $p_i$ is a basis of primitive elements, with $c_{ij}=c_{ji}$. Moreover, $\pi(w)=0$, which implies that $c_{ii}=0$ for all $i$. Now replacing $f$ with $f+\sum_{i<j}c_{ij}p_i\otimes p_j$ (which is possible since this sum is a $2$-cocycle) we come to a situation where $f$ is symmetric, as desired. 
\end{proof}

Choose $f\in (\mathcal{O}(\mathcal{G})^*)^{\ot 2}$ symmetric with the same degree $\ell$ as $\phi$ such that $\phi=df$, which is possible by Lemma \ref{symm1}. Let $\widetilde{f}$ be a symmetric lift of $f$ to $H$. Then the pseudotwist $F:=1+\widetilde{f}$ is symmetric, which implies that $(H,1,\Phi)^{F}=(H^{F},1,\Phi^{F})$, and the pseudotwisted associator $\Phi^{F}$ is equal to $1+$ terms of degree $\ge \ell+1$. By continuing this procedure, we will come to a situation where $(H,1,\Phi)^F=(H^F,1,1)$ for some pseudotwist $F\in H^{\ot 2}$, as desired. This concludes the proof of Theorem \ref{symmwithphimain2} in the case where $R=1$.

\subsubsection{The case $R=1+\dd\ot \dd$ with $\dd\ne 0$.}

In this subsection we will assume that $R=1+\dd\ot \dd$ with $\dd\ne 0$.  
Suppose $\dd$ has degree $p$, and let $\delta$ be its projection to $\gr(H)[p]$.

The following lemma is the analogue of Lemma \ref{qtalt1} in this case.

\begin{lemma}\label{qtalt2a} 
The following hold: 
\begin{enumerate}
\item
$\Delta(\delta)=\delta\ot 1+1\ot \delta$.
\item
The degree of $\Delta(\dd)-\dd\ot 1-1\ot \dd$ is $\ge \ell-p$. 
\item
Let $T\in \gr(H)^{\ot 2}[\ell-p]$ be the part of $\Delta(\dd)-\dd\ot 1-1\ot \dd$ of degree exactly $\ell-p$ (so $T=0$ if $\ell\le 2p$). Then we have 
\begin{enumerate}
\item
$T\ot \delta+\phi_{312}+\phi_{132}+\phi_{123}=0$.
\item
$\delta\ot T+\phi_{231}+\phi_{213}+\phi_{123}=0$.
\item
$\phi_{123}+\phi_{321}=T\ot \delta+\delta\ot T$.
\item
$\Cyc(\phi)=\Cyc(T\otimes \delta)$.
\end{enumerate}
\item 
$T$ is a symmetric $2$-cocycle.
\end{enumerate}  
\end{lemma}

\begin{proof}
(1) is clear. 
(2) and (3) follow immediately from the hexagon relations \cite[(2.6)-(2.7)]{eg5} ((3)(d) is obtained by applying $\Cyc$ to (3)(a) and using that $\Alt(\phi)=0$). 
Also, let $Q:=T+T_{21}$. By (3)(c), we have $Q\otimes \delta=\delta\otimes Q$. Thus both left and right tensorands of $Q$ can only be multiples of $\delta$, i.e., $Q$ is a multiple of $\delta\otimes \delta$. But $\pi(Q)=0$, hence $Q=0$, proving (4).
\end{proof}  

\begin{proposition}\label{killingPhia}
The $3$-cocycle $\phi$ has the form 
$$
\phi=T\otimes \delta+df
$$
for some $f\in k[\mathcal{G}]^{\ot 2}[\ell]$. 
\end{proposition}

\begin{proof} 
By Lemma \ref{qtalt}(1), $\phi\in Z^3(\mathcal{O}(\mathcal{G}),k)$ and we can express it in the following form:
$$
\phi=\sum_{1\le i< j< l\le n}b_{ijl}(y_i\ot y_j\ot y_l)+
\sum_{i,j}a_{ij}\beta_i\otimes y_j + df',
$$
for some $b_{ijl},a_{ij}\in k$ and   
$f'\in k[\mathcal{G}]^{\ot 2}$. 

Since $\Alt(df')=0$, using Lemma \ref{qtalt}(2) this implies that
\begin{equation*}
0=\Alt(\phi)=\sum_{1\le i<j<l\le n}b_{ijl}\Alt(y_i\ot y_j\ot y_l).
\end{equation*}
Therefore $b_{ijl}=0$ for every $i<j<l$. Thus by Lemma \ref{qtalt2a}(2)(c), we have
\begin{eqnarray*}
\Cyc(df') & = & 
\sum_{i,j}a_{ij}\Cyc(\beta_i\otimes y_j)+\Cyc(T\otimes \delta).
\end{eqnarray*}

Now by (\ref{2star}), we have
$$
\Psi(\Cyc(df'))=0.
$$
We also have
$$
\pi(\beta_i)=y_i^{(2^{r_i-1})}.
$$
Thus,
\begin{equation}\label{3star}
\sum_{i,j}a_{ij}y_i^{(2^{r_i-1})}\otimes y_j+\pi(T)\otimes \delta=0,
\end{equation}
which implies that 
$$
\sum_j a_{ij}y_j=a_i\delta,\,\,\,\text{and}\,\,\,\sum_i a_iy_i^{(2^{r_i-1})}=\pi(T)
$$
for some $a_i\in k$. 
Hence, 
$$
\sum a_i\beta_i+T=dh 
$$
for some $h\in k[\mathcal{G}]$ (as the left hand side is a symmetric $2$-cocycle killed by $\pi$, hence a coboundary). So, 
$$
\sum a_i\beta_i\otimes \delta=T\otimes \delta+d(h\otimes \delta)
$$ 
(as $d\delta=0$). This implies that $\phi=T\otimes \delta+df$, where $f:=f'+h\otimes \delta$, as desired.  
\end{proof}

\begin{lemma}\label{symm} 
In Proposition \ref{killingPhia} we can choose $f\in \Gamma^2k[\mathcal{G}]$, i.e., we can choose $f$ to be symmetric. 
\end{lemma} 

\begin{proof} Since $\phi_{123}=\phi_{321}+T\otimes \delta+\delta\otimes T$ by Lemma \ref{qtalt2a}(3)(c), we have $df=d(f_{21})$. Thus, $f+f_{21}\in Z^{2}(\mathcal{O}(\mathcal{G}),k)$ is a $2$-cocycle, and we can proceed in exactly the same way as in the proof of Lemma \ref{symm1} to get to a situation where $f$ is symmetric. 
\end{proof} 

\begin{proposition} \label{fourcoc} 
The $4$-cocycle $T\otimes T$ is a coboundary. 
\end{proposition} 

\begin{proof} 
Let $f$ be a symmetric element provided by Lemma \ref{symm}, 
and let $\widetilde{f}$ be a symmetric lift of $f$ to $H$. 
Then the pseudotwist $F:=1+\widetilde{f}$ is symmetric.
Thus, $(H,R,\Phi)^{F}=(H^{F},R^F,\Phi^{F})$, and $\Phi^F-1$ has degree $\ge \ell$ with degree $\ell$ part \linebreak $T\otimes \delta$. Thus, we have 
$$
\Phi^F=1+(\Delta(\dd)-\dd \ot 1-1\ot \dd)\ot\delta+U,
$$ 
where $U\in H^{\ot 3}$ has degree $\ge \ell+1$. 
 
The pentagon equation \cite[(2.3)]{eg5} for $\Phi^F$ yields that $dU$ has degree $\ge 2\ell-2p$, and its part of degree $2\ell-2p$ is $T\otimes T$. This means that $U$ has degree $s\le 2\ell-2p$. Let $u$ be the leading part of $U$. If $s<2\ell-2p$ then the pentagon equation \cite[(2.3)]{eg5} yields that $du=0$, and arguing as above we see that $u=df$, where $f$ is symmetric. Thus, by a gauge transformation, we can make sure that $u=0$. Thus, we may assume that $s=2\ell-2p$. In this case \cite[(2.3)]{eg5} yields $du=T\otimes T$, i.e., $T\otimes T$ is a coboundary, as claimed. 
\end{proof} 

\begin{proposition}\label{phicobrd} 
The $3$-cocycle $\phi$ is a coboundary. 
\end{proposition} 

\begin{proof} 
By \cite[Proposition 2.4(2)]{eg5} on the structure of cohomology, $\pi(T)=0$. Thus by (\ref{3star}), $a_{ij}=0$ for all $i,j$, so $\phi$ is a coboundary.  
\end{proof} 

We can now proceed as in the case $R=1$. Namely, by Proposition \ref{phicobrd}, we have 
$\phi=df$ for some $f\in (\mathcal{O}(\mathcal{G})^*)^{\ot 2}$ with the same degree $\ell$ as $\phi$, and by Lemma \ref{symm}, we can choose $f$ to be symmetric. Then letting $\widetilde{f}$ be a symmetric lift of $f$ to $H$, we get the symmetric pseudotwist $F:=1+\widetilde{f}$, and by this pseudotwist we come to the situation where $\Phi-1$ has degree $\ge \ell+1$. Thus $\Delta(\dd)-\dd\otimes 1-1\otimes \dd$ also has degree $\ge \ell+1$. 

However, unlike in the case $R=1$, we are not done yet since the pseudotwist $F$ spoils 
the $R$-matrix. Namely, since $f$ is symmetric, $R$ has been brought to the form
$$
R=1+\dd\ot \dd+[\dd\ot \dd,f]+{\rm terms\,\,of\,\,degree\,\,> 2\ell-2p}. 
$$
Thus, we need the following lemma.

\begin{lemma}\label{last}
We can twist further to make sure that $R=1+\dd\otimes \dd$ and still $\Phi-1$ has degree $\ge \ell+1$.
\end{lemma}

\begin{proof}
Let $v:=\pi(f)$. Then $f=v\otimes v+h+h_{21}$ for some $h\in k[\mathcal{G}]^{\ot 2}$. 
Thus, by twisting by the pseudotwist $J:=1+[\dd\otimes \dd,h]+\dd v\otimes v\dd$, we come to the situation where $\Phi-1$ still has degree 
$\ge \ell+1$, but 
$$
R=1+\dd\ot \dd+[\dd,v]\ot [\dd,v]+{\rm terms\,\,of\,\,degree\,\,> 2\ell-2p}. 
$$

Now, if $\ell<4p$ then $\ell/2+2p>\ell$, so twisting by $J:=1+\dd\otimes [\dd,v]$, we get to a situation when $\Phi-1$ is of degree $\ge \ell+1$ and $$
R=1+\dd\ot \dd+{\rm terms\,\,of\,\,degree\,\,> 2\ell-2p}. 
$$ 
Now Proposition \ref{trivR} implies that using twists $J$ with $J-1$ of degree $\ge \ell+1$ we can come to a situation where $\Phi=1$ modulo degree $\ge \ell+1$ and $R=1+\dd\ot \dd$ on the nose, providing the desired induction step. 

It remains to consider the situation $\ell\ge 4p$. By twisting by $J:=1+\dd\otimes [\dd,v]$, we will get to a situation where $\Phi-1=\dd\otimes W$ + terms of degree $\ge \ell+1$ and $R=1+\dd\ot \dd + {\rm terms\,\,of\,\,degree\,\,> 2\ell-2p}$, where $$W:=\Delta([\dd,v])+[\dd,v]\ot 1+1\ot [\dd,v].$$ If 
$\deg(W)>\ell-p$ then we are done with the induction step, so it remains to consider the case 
$\deg(W)\le \ell-p$. In this case the hexagon relations \cite[(2.6)-(2.7)]{eg5} yield $W=0$. 
Thus we come to a situation where $\Phi-1$ has degree $\ge \ell+1$ and $R-1-\dd\ot \dd$ has degree $> 2\ell-2p$. So by Proposition \ref{trivR}, by applying twists of degree $>\ell$, we can make sure that $R=1+\dd\otimes \dd$ and still $\Phi-1$ has degree $\ge \ell+1$, as desired.
\end{proof}

Thus it follows from the above that by continuing this procedure, we will come to a situation where $$(H,1+\dd\ot \dd,\Phi)^F=(H^F,1+\dd\ot \dd,1)$$ for some pseudotwist $F\in H^{\ot 2}$, as desired. This concludes the proof of Theorem \ref{symmwithphimain2}
in the case where $R=1+\dd\ot \dd$.

The proofs of Theorems \ref{symmwithphimain2} and  \ref{classuniptr2} are complete. \qed

\begin{remark} 
Here is another short proof of the case when $R$ is twist equivalent to $1$, which uses the result of Coulembier. If $R=1$ then the symmetric square of a representation $V$ is the usual one, so for any injection 
$k\to V$ the induced map $k\to S^2V$ is injective. By \cite[Theorem C]{c}, this implies
that the category ${\rm Rep}(H,1,\Phi)$ is locally semisimple. Hence by \cite[Proposition 6.2.2]{c}, the maximal Tannakian subcategory of ${\rm Rep}(H,1,\Phi)$ is 
a Serre subcategory. Since the subcategory of ${\rm Rep}(H,1,\Phi)$ generated by simple objects is Tannakian, we see that the whole category ${\rm Rep}(H,1,\Phi)$ is Tannakian, which implies the desired statement. 
\end{remark} 

\begin{remark} 
The case when $R\ne 1$ is more subtle, as it is not captured by first order deformation theory. Indeed, the category $\mathcal{D}={\rm Rep}(k[\dd]/(\dd^2),1+\dd\otimes \dd)$ 
has a nontrivial first order deformation over $k[h]/(h^2)$, with the same $R$-matrix $R$, but with 
$\Delta(\dd)=\dd\otimes 1+1\otimes \dd+h\dd\otimes \dd$ and associator $\Phi:=1+h\dd\otimes \dd\otimes \dd$. 
This deformation is nontrivial because $\phi:=\dd\otimes \dd\otimes \dd$ is a nontrivial $3$-cocycle. However, it does not lift to $k[h]/(h^3)$, as the difference between the left hand side and the right hand side of the pentagon equation \cite[(2.3)]{eg5} is $h^2\dd^{\otimes 4}$.

The existence of such deformations is typical. For example, consider the category $\Vect(\mathbb{Z}/p\mathbb{Z})$ in characteristic $p>0$. Clearly, it has no nontrivial formal deformations, since $H^3(\mathbb{Z}/p\mathbb{Z},k^{\times})$ is trivial. However, it has a nontrivial first order deformation, since $H^3(\mathbb{Z}/p\mathbb{Z},k)=k$. This deformation in fact lifts modulo $h^i$ for any $i\le p$, but does not lift modulo $h^{p+1}$. This is because $\mu_p$ and $\alpha_p$ are ``the same" up to order $p-1$ inclusively, but differ in order $p$. 
\end{remark}

\begin{corollary}\label{symchevprop}
Let $(H,R)$ be a finite dimensional triangular Hopf algebra with the Chevalley property over $k$. Then $(H,R)$ is twist equivalent to a triangular Hopf algebra with $R$-matrix $1+\dd\ot \dd$ for some $\dd\in P(H)$ such that $\dd^2=0$. 
\end{corollary}

\begin{proof}
Applying Theorem \ref{symmwithphimain2} to $(H,R,1)$ yields the existence of a pseudotwist $J$ for $H$ such that $(H,R,1)^J=(H^J,1+\dd\ot \dd,1)$. In particular, we have $1^J=1$, which is equivalent to $J$ being a twist.
\end{proof}

\begin{corollary}\label{alphapsym2}
Let $\mathcal{C}$ be a finite symmetric tensor category over $k$ such that $\FPdim(\mathcal{C})=2$. Then $\mathcal{C}$ is symmetric tensor equivalent to either $\Vect(\mathbb{Z}/2\mathbb{Z})$, $\Rep(\mathbb{Z}/2\mathbb{Z})$, $\Rep(\alpha_2)$ or $\mathcal{D}$.
\end{corollary}

\begin{proof}
Follows immediately from Theorem \ref{classuniptr2}.
\end{proof}

\subsection{Strengthening of \cite[Theorem 1.1]{eg5} and Theorem \ref{classuniptr2}}

The arguments used in this section and \cite[Section 3]{eg5} in fact prove a stronger result. Namely, we have the following theorem.

\begin{theorem}\label{stronger}
Let $\mathcal{E}\subset \mathcal{C}$ be finite symmetric tensor categories over an algebraically closed field $k$ with characteristic $p>0$, such that $\mathcal{E}$ contains all the simples of $\mathcal{C}$. The following hold:
\begin{enumerate}
\item
Suppose $p>2$. If $\mathcal{E}$ has a fiber functor to ${\rm sVec}$, then so does $\mathcal{C}$. 

\item
Suppose $p=2$. If $\mathcal{E}$ has a fiber functor to ${\rm Vec}$, then $\mathcal{C}$ has a fiber functor to $\mathcal{D}$.
\end{enumerate}
\end{theorem}

Indeed, in both cases it follows that $\mathcal{C}$ is integral, so we have $\mathcal{C}=\Rep(H,R,\Phi)$ for some finite dimensional triangular quasi-Hopf algebra over $k$. Now the arguments are exactly the same, except the radical of $H$ should be replaced by the annihilator of $\mathcal{E}$ inside $\mathcal{C}$, which is a nilpotent quasi-Hopf ideal of $H$ since $\mathcal{E}$ contains all the simples of $\mathcal{C}$.  

\section{Twists and Sweedler cohomology for finite abelian $p$-groups}

In this section we let $K$ be an arbitrary field of characteristic $p>0$, and $\mathbb{F}_q$ be a finite field of characteristic $p>0$.

\subsection{Truncated Witt vectors} Let $W_n(K)$ be the {\em ring of truncated Witt vectors of length $n$ with coefficients in $K$}. Recall that $W_n(K)= K^n$ as a set, with nontrivial addition and multiplication given, e.g., in \cite[VI, p.330-332]{l}. 
\begin{example}
We have the following:
\begin{enumerate} 
\item
$W_1(K)=K$ as rings.

\item
The addition and multiplication in $W_2(K)$ are given as follows
$$
(x_0,x_1)+(y_0,y_1)=\left(x_0+y_0,x_1+y_1+\sum_{i=1}^{p-1}\frac{1}{i}\binom{p-1}{i-1}x_0^i y_0^{p-i}\right)
$$
and
$$
(x_0,x_1)(y_0,y_1)=\left(x_0y_0,y_0^px_1+y_1x_0^p\right).
$$

\item
$W_n(\mathbb{F}_p)=\mathbb{Z}/p^n\mathbb{Z}$ for every $n\ge 1$.
\end{enumerate}
\end{example} 

For $x:=(x_0,\dots,x_{n-1})\in W_n(K)$, let $F(x)=(x_0^p,\dots,x_{n-1}^p)$. (Note that if $n>1$ then $F(x)\ne x^p$.) Recall that $F:W_n(K)\to W_n(K)$ is a ring homomorphism, and we have an additive homomorphism 
$$\mathscr{P}:W_n(K)\to W_n(K),\,\,\,x\mapsto F(x)-x.$$  
The kernel of $\mathscr{P}$ is the cyclic group $W_n(\mathbb{F}_p)=\mathbb{Z}/p^n\mathbb{Z}$.

\begin{lemma}\label{wittfq}
The following hold:
\begin{enumerate}
\item
If $K$ is perfect then $W_n(K)/\mathscr{P}(W_n(K))$ is a free $\mathbb{Z}/p^n\mathbb{Z}$-module.
\item
$W_n(\mathbb{F}_q)/\mathscr{P}(W_n(\mathbb{F}_q))\cong \mathbb{Z}/p^n\mathbb{Z}$.
\end{enumerate}
\end{lemma}

\begin{proof}
(1) First note that since $K$ is perfect, we have $W_n(K)/p^sW_n(K)\cong W_s(K)$ for every $0\le s\le n$. 

Secondly, let $a\in W_n(K)$ be an element such that its image $a_0$ in $K$ is not in $\mathscr{P}(K)$. We claim that the order of $a$ in $W_n(K)/\mathscr{P}(W_n(K))$ is $p^n$. Indeed, suppose $s<n$ is such that $p^sa=0$ in $W_n(K)/\mathscr{P}(W_n(K))$, i.e., $p^sa=\mathscr{P}(y)$ for some $y\in W_n(K)$. Then $\mathscr{P}(y)=0$ in $W_n(K)/p^sW_n(K)=W_s(K)$. Thus $y=k\in \mathbb{Z}/p^s\mathbb{Z}\subseteq W_n(K)/p^sW_n(K)$ (as ${\rm ker}(\mathscr{P})=\mathbb{Z}/p^n\mathbb{Z}$), so $y=k+p^sz$ for some integer $k$ and $z\in W_n(K)$. But then $p^sa=\mathscr{P}(y)=\mathscr{P}(p^sz)$, so if $z_0$ is the image of $z$ in $K$ then $a_0=\mathscr{P}(z_0)$, which is a contradiction.

Finally, take $a\in W_n(K)$ such that $p^{n-1}a=0$ in $W_n(K)/\mathscr{P}(W_n(K))$, and consider its image $a_0$ in $K$. We have shown that $a_0$ must be in $\mathscr{P}(K)$, i.e., $a_0=x_0^p-x_0$ for some $x_0$ in $K$. Let $x:=(x_0,0,\dots,0)\in W_n(K)$. We have $a-\mathscr{P}(x)=py$ for some $y\in W_n(K)$ (again using that $K$ is perfect). Thus $a=py$ in $W_n(K)/\mathscr{P}(W_n(K))$, proving freeness.  

(2) Since the kernel of $\mathscr{P}:W_n(\mathbb{F}_q)\to W_n(\mathbb{F}_q)$ is 
$\mathbb{Z}/p^n\mathbb{Z}$, it follows that the cokernel of 
$\mathscr{P}$ has order $p^n$. Thus $W_n(\mathbb{F}_q)/\mathscr{P}(W_n(\mathbb{F}_q))$ is abelian of order $p^n$, so the claim follows from Part (1).
\end{proof}

\begin{remark}
If $K$ is not perfect then for instance $W_2(K)$ is not a free $\mathbb{Z}/p^2\mathbb{Z}$-module. Indeed, take an element $(0,a)$ in $W_2(K)$, where $a\in K$ is not a $p$th power. Then $p(0,a)=(0,1)(0,a)=0$, but $(0,a)\ne p(x,y)$ for any $x,y$, since $p(x,y)=(0,x^p)$.
\end{remark}

\subsection{Twists for abelian groups and torsors} Recall that an interesting invariant of a tensor category $\mathcal{C}$ over $K$ is the group of tensor structures on the identity functor of $\mathcal{C}$ (i.e., the group of isomorphism classes of tensor autoequivalences of $\mathcal{C}$ which act trivially on the underlying abelian category) up to an isomorphism \cite{da,bc}. This group is called {\em the second invariant (or lazy) cohomology group} of $\mathcal{C}$ and denoted by $H^2_{\rm inv}(\mathcal{C},K)$. 

In particular, if $\mathcal{C}:={\rm Rep}_K(A)$ is the representation category of a finite abelian group $A$ then $H^2_{\rm inv}(A,K):=H^2_{\rm inv}(\mathcal{C},K)$ is the group of gauge equivalence classes of twists for the Hopf algebra $K[A]$ \cite{eg2}. 

\begin{lemma}\label{torsor}
Let $A$ be a finite abelian group. We have a canonical group isomorphism $H^2_{\rm inv}(A,K)\cong {\rm Hom}(G,A)$, where $G:={\rm Aut}(\overline{K}/K)={\rm Gal}(K^s/K)$. \footnote{When considering {\rm Hom} from a profinite group, as usual it means continuous homomorphisms.}
\footnote{$K^s$ is the separable closure of $K$.}
\end{lemma}

\begin{proof}
Let $J$ be a twist for $K[A]$, and consider the twisted $K$-algebra $(K[A]_J)^*$. Observe that (up to $K$-algebra isomorphism) this algebra depends only on $[J]$. Since by \cite[Theorem 6.5]{aegn} every twist for $\overline{K}[A]$ is trivial, it follows that $(K[A]_J)^*\otimes_K \overline{K}$ and ${\rm Fun}(A,\overline{K})$ are isomorphic as $\overline{K}$-algebras. Thus, $(K[A]_J)^*$ is a semisimple commutative $K$-algebra. Furthermore, $(K[A]_J)^*$ is an $A$-algebra, which is isomorphic to the regular representation of $A$ as an $A$-module. Thus $(K[A]_J)^*$ is an $A$-torsor.

Conversely, suppose $B$ is an $A$-torsor, i.e., a commutative semisimple $K$-algebra with an $A$-action such that $B\ot _K \overline{K}\cong {\rm Fun}(A,\overline{K})$. By Wedderburn theorem, $B$ decomposes uniquely into a direct sum of field extensions $L_i$ of $K$: $B=\bigoplus_i L_i$. Since the space of $A$-invariants in $B$ is $1$-dimensional, $A$ acts transitively on the set of fields $L_i$. Let $H\subseteq A$ be the stabilizer of $L:=L_1$. Clearly $L$ is a cyclic extension of $K$ with Galois group $H$. Then it is well known that $L\cong (K[H]^*)_J$ for a unique (up to gauge equivalence) Hopf $2$-cocycle $J$ for $K[H]^*$. Viewing $J$ as a twist for $K[H]$ (hence for $K[A]$), it is easy to see that the class $[J]$ is uniquely determined by the isomorphism class of the $A$-torsor $B$.

Finally we note that $A$-torsors form an abelian group under the product rule $(B_1,B_2)\mapsto (B_1\otimes B_2)^A$, where $a\in A$ acts on $B_1$ by $a$ and on $B_2$ by $a^{-1}$, and that $(K[A]_{IJ})^*\cong ((K[A]_I)^*\otimes (K[A]_J)^*)^A$ (see, e.g., \cite[Remark 3.12]{aegn}).

It now follows from the above that the group $H^2_{\rm inv}(A,K)$ is canonically isomorphic to the group of $A$-torsors over $K$. Since the latter is canonically isomorphic to the Galois cohomology group $H^1(G,A)=\Hom(G,A)$, the claim follows.
\end{proof} 

\subsection{Invariant cohomology of abelian groups} Let $A$ be a finite abelian group of exponent dividing $p^n$. Let $G$ be as in Section 3.2, and let $G_n$ be its maximal abelian quotient of exponent dividing $p^n$. Then $\Hom(G,A)=\Hom(G_n,A)$. Thus by Lemma \ref{torsor}, we have a canonical group isomorphism 
\begin{equation}\label{gn}
H^2_{\rm inv}(A,K)\cong \Hom(G_n,A).
\end{equation}

\begin{theorem}\label{invtwpabconc}
Let $A$ be a finite abelian group of exponent dividing $p^n$. Then the following hold:
\begin{enumerate}
\item
We have a canonical group isomorphism 
$$H^2_{\rm inv}(A,K)\cong {\rm Hom}(A^{\vee},W_n(K)/\mathscr{P}(W_n(K))),$$ where $A^{\vee}:={\rm Hom}(A,\mathbb{Z}/p^n\mathbb{Z})$.
\item
If moreover $K$ is perfect then we have a canonical group isomorphism 
$$H^2_{\rm inv}(A,K)\cong A\otimes_{\mathbb{Z}/p^n\mathbb{Z}} (W_n(K)/\mathscr{P}(W_n(K))).$$
\end{enumerate}
\end{theorem}

\begin{proof}
(1) Recall that Artin-Schreier-Witt theory provides a canonical group isomorphism $$G_n\xrightarrow{\cong} \Hom(W_n(K)/\mathscr{P}(W_n(K)),\mathbb{Z}/p^n\mathbb{Z})$$ (see, e.g., \cite[VI, p.330--332]{l}). Thus we get from (\ref{gn}) a canonical group isomorphism
$$H^2_{\rm inv}(A,K)\cong \Hom(\Hom(W_n(K)/\mathscr{P}(W_n(K)),\mathbb{Z}/p^n\mathbb{Z}),A).$$
The claim follows now from the fact that ${\rm Hom}(B^{\vee},A)={\rm Hom}(A^{\vee},B)$ for every $B$.

(2) By Lemma \ref{wittfq}(1), $W_n(K)/\mathscr{P}(W_n(K))$ is a free $\mathbb{Z}/p^n\mathbb{Z}$-module. Therefore the group 
$${\rm Hom}(A^{\vee},W_n(K)/\mathscr{P}(W_n(K)))\cong \Hom(\Hom(W_n(K)/\mathscr{P}(W_n(K)),\mathbb{Z}/p^n\mathbb{Z}),A)$$ is the same as the group $A\otimes_{\mathbb{Z}/p^n\mathbb{Z}} (W_n(K)/\mathscr{P}(W_n(K)))$, as desired.
\end{proof}

\begin{corollary}\label{twistsabelian}
We have a group isomorphism 
$$H^2_{\rm inv}(\mathbb{Z}/p^n\mathbb{Z},K)\cong W_n(K)/\mathscr{P}(W_n(K)).$$
In particular, we have a group isomorphism 
$$H^2_{\rm inv}(\mathbb{Z}/p^n\mathbb{Z},\mathbb{F}_q)\cong \mathbb{Z}/p^n\mathbb{Z}.$$
\end{corollary}

\begin{proof}
By Theorem \ref{invtwpabconc}(1), $H^2_{\rm{inv}}(\mathbb{Z}/p^n\mathbb{Z},\mathbb{F}_q)\cong W_n(\mathbb{F}_q)/\mathscr{P}(W_n(\mathbb{F}_q))$, so the second claim follows from Lemma \ref{wittfq}(2).
\end{proof}

\begin{remark}
(1) Theorem \ref{invtwpabconc}(1) implies that if $K$ is algebraically closed then $H^2_{\rm{inv}}(A,K)=0$, which agrees with \cite[Proposition 5.7]{eg5} for $i=2$.

(2) Theorem \ref{invtwpabconc}(1) was obtained by Guillot \cite{gu} for $p=2$ and $n=1$. 
\end{remark}

\subsection{Sweedler cohomology of algebras of functions on abelian groups} Let $A$ be a finite abelian group, and let $\mathcal{O}(A)$ be the Hopf algebra of functions on $A$ with values in $K$. Recall that $H^2_{\rm{inv}}(A,K)$ coincides with the second Sweedler cohomology group $H^2_{\rm{Sw}}(\mathcal{O}(A),K)$ with coefficients in $K$. 

\begin{theorem}\label{sweedcoh}
Let $A$ be a finite abelian group of exponent dividing $p^n$. Then the Sweedler cohomology of $\mathcal{O}(A)$ with coefficients in $K$ is as follows: 
\begin{enumerate}
\item
$H^1_{\rm Sw}(\mathcal{O}(A),K)=A$.
\item
$H^2_{\rm Sw}(\mathcal{O}(A),K)={\rm Hom}(A^{\vee},W_n(K)/\mathscr{P}(W_n(K)))$.
\item
$H^i_{\rm Sw}(\mathcal{O}(A),K)=0$ for every $i\ge 3$.
\end{enumerate} 
\end{theorem}

\begin{proof}
(1) is clear and (2) is Theorem \ref{invtwpabconc}(1). To prove (3) consider the normalized complex computing $H^i_{\rm Sw}(\mathcal{O}(A),K)$: 
$$C^0(K)\to C^1(K)\to C^2(K)\to \cdots,$$ where 
$C^i$ is the algebraic group such that for any field $L$, $C^i(L)=(L[A]^{\otimes i})^{\times}_1$ is the group of invertible elements $a$ in $L[A]^{\otimes i}$ with 
$\varepsilon(a)=1$. Then $C^i$ is a connected commutative unipotent algebraic group over $K$ (i.e., an iterated extension of $\mathbb{G}_a$). 

Now fix $n\ge 2$. Since by \cite[Proposition 5.7]{eg5}, $$H^n_{\rm Sw}(\mathcal{O}(A),\overline{K})=H^{n+1}_{\rm Sw}(\mathcal{O}(A),\overline{K})=0,$$ we have a short exact sequence 
$$0\to C^{n-1}/D^{n-1}\to C^n\to D^{n+1}\to 0,$$ where $D^{i}\subseteq C^i$ is the kernel of the differential map $d:C^{i}\to C^{i+1}$. Thus we have an exact sequence 
$$0\to (C^{n-1}/D^{n-1})(K)\to C^n(K)\to D^{n+1}(K)\to H^1(K,C^{n-1}/D^{n-1}),$$ 
where $$H^1(K,C^{n-1}/D^{n-1}):=H^1({\rm Gal}(\overline{K}/K),(C^{n-1}/D^{n-1})(\overline{K}))$$ is the Galois cohomology group. 
But since $C^{n-1}/D^{n-1}$ is an iterated extension of $\mathbb{G}_a$, and $H^1(K,\mathbb{G}_a)=0$, the Galois cohomology group $H^1(K,C^{n-1}/D^{n-1})$ vanishes. Thus we have a short exact sequence 
$$0\to (C^{n-1}/D^{n-1})(K)\to C^n(K)\to D^{n+1}(K)\to 0,$$ which implies that $H^{n+1}_{\rm Sw}(\mathcal{O}(A),K)=D^{n+1}(K)/d(C^n(K))=0$, as claimed.
\end{proof}

\end{document}